\documentclass[a4paper,reqno]{amsart}
\usepackage{a4wide,mathptmx} 
\usepackage{amssymb,extarrows}

\providecommand{\cal}[1]{\mathcal{#1}}

\newcommand{\gen}{\mathbf{A}}

\newcommand{\B}{{\mathbb{B}}}
\newcommand{\C}{{\mathbb{C}}}
\newcommand{\N}{{\mathbb{N}}}
\newcommand{\R}{{\mathbb{R}}}
\newcommand{\Rn}{{\mathbb{R}}^{n}}

\newcommand{\dv}{\operatorname{div}}
\newcommand{\grad}{\operatorname{grad}}
\newcommand{\im}{\operatorname{i}}
\newcommand{\lap}{\operatorname{\Delta}}
\newcommand{\mlap}{-\!\operatorname{\Delta}}
\newcommand{\scal}[2]{(\,#1\,|\, #2\,)}
\newcommand{\set}[2]{\{\,#1\bigm| #2\,\}}
\newcommand{\Set}[2]{\bigl\{\,#1\bigm| #2\,\bigr\}}
\newcommand{\vvvert}{{|\hspace{-1.6pt}|\hspace{-1.6pt}|}}
\newcommand{\dual}[2]{\ensuremath{\langle{#1},{#2}\rangle}}

\newtheorem{theorem}{Theorem}
\newtheorem{proposition}{Proposition}
\newtheorem{lemma}{Lemma}
\newtheorem{corollary}{Corollary}

\theoremstyle{definition}

\theoremstyle{remark}
\newtheorem{remark}{Remark}

\begin{document}
\title[Isomorphic well-posedness of the
  inverse Neumann heat equation]{Isomorphic well-posedness of the
  final value problem\\ for the heat equation with the homogeneous Neumann condition}

\author[J.~Johnsen]{Jon Johnsen}
\address{Department of Mathematics, Aarhus  University, Ny Munkegade 118, Building 1530, DK-8000 Aarhus C, Denmark}
\email{jojoh@math.au.dk}

\begin{abstract} This paper concerns the final value problem for the heat
  equation subjected to the homogeneous Neumann condition on the boundary of a smooth open set in
  Euclidean space. The problem is here shown to be isomorphically well posed in the sense that
  there exists a linear homeomorphism between suitably chosen Hilbert spaces containing the
  solutions and the data, respectively. This improves a recent work of the author, in which
  the same problem was proven well-posed in the original sense of Hadamard under an additional
  assumption of H\"older continuity of the source term. Like for its predecessor,
  the point of departure is an abstract analysis in spaces of vector distributions of final
  value problems generated by coercive Lax--Milgram operators, now yielding isomorphic well-posedness
  for such problems.
  Hereby the data space is the graph normed domain of an unbounded operator that maps final states
  to the corresponding initial states, resulting in a non-local compatibility condition on the data.
  As a novelty, a stronger version of the compatibility condition is introduced with the purpose of
  characterising  the data that yield solutions having the regularity property of being square
  integrable in the generator's graph norm (instead of in the form domain norm). This
  result allows a direct application to the class 2 boundary condition in the considered inverse
  Neumann heat problem.
\end{abstract}


\maketitle

\section{Introduction}
The purpose of the present paper is to show rigorously that the heat conduction final value  problem 
with the homogeneous Neumann condition is
\emph{isomorphically} well-posed, in the sense that there exists an isomorphism between suitably
chosen spaces for the data and the corresponding solutions. 

This result is obtained below by improving the recent results in \cite{JJ19cor} by application of classical regularity properties in spaces of low regularity.

\bigskip

The central theme below is to characterise the 
$u(t,x)$ that in a fixed bounded open set 
$\Omega\subset\Rn$ ($n\ge1$) with $C^\infty$-smooth boundary $\Gamma=\partial\Omega$
fulfil the following equations, whereby $\Delta=\partial_{x_1}^2+\dots+\partial_{x_n}^2$ denotes
the Laplace operator and $\nu(x)$ stands for the exterior unit normal vector field at $\Gamma$:
\begin{equation}  \label{heat-intro}
\left.
\begin{aligned}
  \partial_tu(t,x)-\lap u(t,x)&=f(t,x) &&\quad\text{for $t\in\,]0,T[\,$,  $x\in\Omega$},
\\
   (\nu\cdot\grad) u(t,x)&=0 &&\quad\text{for $t\in\,]0,T[\,$, $x\in\Gamma$},
\\
  u(T,x)&=u_T(x) &&\quad\text{for $t=T$, $x\in\Omega$}.
\end{aligned}
\right\}
\end{equation}
In view of the final value condition at $t=T$, this problem is also called the inverse Neumann heat equation.
One area of interest of this could be a nuclear power plant 
hit by power failure at $t=0$: when at $t=T>0$ power is regained and 
the reactor temperatures $u_T(x)$ are measured, a
backwards calculation could possibly settle whether at some $t<T$ the temperatures $u(t,x)$
were high enough to cause damage to the fuel rods. 

Here it should be noted that the Neumann condition, which controls the heat \emph{flux} through the
boundary, is more natural from a physical point of view than the Dirichlet condition
$u|_{\Gamma}=g$, in which the value itself is prescribed at the boundary in terms of
a given function $g(t,x)$ on $\,]0,T[\,\times\Gamma$. Other boundary conditions may be also be natural, but are for the sake of simplicity not treated here.

Previously, the final value problem for the heat equation with the Dirichlet condition was shown to
be well posed in a joint work of the author \cite{ChJo18ax}, with more concise expositions in
\cite{ChJo18, JJ19}. 
For this problem, the obtained well-posedness was in the original sense of Hadamard, namely, that
there is
\emph{existence, uniqueness} and \emph{stability} of a solution $u\in X$ for given data
$(f,g,u_T)\in Y$, in certain Hilbertable spaces $X$, $Y$ that were described explicitly. 
Hereby the data space $Y$ was defined in terms of a special compatibility condition on the triples
$(f,g,u_T)$, which was introduced for the purpose in \cite{ChJo18ax}. 

This development seemingly closed a gap that had remained in the understanding since the 1950's,
even though well-posedness is crucial for the interpretation and accuracy of 
numerical schemes for the problem (the work of John~\cite{John55} was pioneering,
but e.g.\ also Eld\'en \cite{Eld87} could be mentioned).
Briefly phrased, the results are obtained via a suitable structure on the reachable set for 
parabolic evolution equations.

\bigskip

In the present paper the intention is not just to focus on the more relevant Neumann condition, but rather
to go an important step further by introducing solution and data spaces $X_1$ and $Y_1$ that are so
chosen that the operator ${\cal P}(u)=(u'-\lap u, u(T))$ is an isomorphism, that is, a linear homeomorphism
\begin{equation}
  X_1 \xleftrightarrow{\;\;{\cal P}\;\;} Y_1.
\end{equation}
It is also proposed to indicate this
strong form of well-posedness by terming \eqref{heat-intro} \emph{isomorphically well posed}; 
cf.\ the title of the paper.

More specifically, using the full yield of the source term, which is the vector $y_f=\int_0^T e^{(T-t)\lap_N}f(t)\,dt$,
the isomorphic well-posedness of \eqref{heat-intro} is obtained below for the spaces 
\begin{align}
  X_1 &= L_2(0,T;H^2(\Omega)) \bigcap C([0,T]; H^1(\Omega)) \bigcap H^1(0,T; L_2(\Omega)),
   \label{X1-intro}
\\
  Y_1&= \left\{ (f,u_T) \in L_2(0,T;L_2({\Omega})) \oplus H^1(\Omega) \Bigm|  
                  u_T - y_f \in e^{T\lap_N}\big(H^1(\Omega)\big) \right\},
\label{Y1-intro}
\end{align}
which are both shown to be Banach spaces under the norms
\begin{align}
  \|u\|_{X_1}&= \Big(\int_0^T\|u(t)\|^2_{H^{2}(\Omega)}\,dt 
                 +\sup_{t\in[0,T]}\sum_{|\alpha|\le 1} \int_\Omega |\partial^\alpha_x u(x,t)|^2\,dx+
                 \int_0^T\|\partial_t u(t)\|^2_{L_2({\Omega})}\,dt  \Big)^{1/2},
\\  \label{Y1norm-intro}
 \| (f,u_T) \|_{Y_1}  
  &= \Big(\int_0^T\|f(t)\|^2_{L_2({\Omega})}\,dt 
  + \int_\Omega\big(\sum_{|\alpha|\le1}|\partial^\alpha_x u_T(x)|^2+\sum_{|\alpha|\le1}|\partial^\alpha_x e^{-T\lap_N}(u_T - y_f )(x)|^2\big)\,dx\Big)^{1/2}.
\end{align} 
Furthermore, it is also established that the solution  is given for $0\le t\le T$ 
by the following variant of the Duhamel formula,
\begin{equation} \label{Duhamel-intro}
  u(t) = e^{t\lap_N}e^{-T\lap_N}\Big(u_T-\int_0^T e^{(T-t)\lap_N}f(t)\,dt\Big) + \int_0^t e^{(t-s)\lap_N}f(s) \,ds.
\end{equation}
For brevity the reader is referred to the details in Theorem~\ref{heatN'-thm} and
Corollary~\ref{heatN-cor} below on the
isomorphic well-posedness of \eqref{heat-intro}. These results substantially improve the Hadamard well-posedness in \cite{JJ19cor}.

However, it is noteworthy here, that both formula \eqref{Duhamel-intro} as well as the definition
of $Y_1$ and its norm in \eqref{Y1-intro}, \eqref{Y1norm-intro} make use of the 
fact that the Neumann realisation $\lap_N$ of the Laplace operator generates an analytic semigroup
$e^{T\lap_N}$ in $L_2(\Omega)$, which is \emph{invertible} in the class of closed operators in
$L_2(\Omega)$, so that one can set
\begin{equation} \label{inv-id}
  e^{-T\lap_N}=(e^{T\lap_N})^{-1}.  
\end{equation}
Section~\ref{inj-sect} below reviews the fact that every analytic semigroups consists entirely of injective
operators, which has \eqref{inv-id}  as a special case. The exposition there is close to
the account in \cite{ChJo18ax}, but it has been included not only to make the
present paper reasonably self-contained, but also to make it precise that the semigroups need not
be uniformly bounded (following up on the indications made in \cite{JJ19,JJ19cor})
and to add a local version and some historical remarks.

As prerequisites, Section~\ref{ivp-sect} recalls a few known extensions of the  analysis
of initial value problems in the classical treatise of Lions and Magenes \cite{LiMa72}. Some of this
was detailed in \cite{JJ19cor}, like solvability theory of problems generated by
$V$-coercive Lax--Milgram operators, estimates of the solution operator and the resulting
Duhamel formula. But for the Neumann problem it is indispensable to include a regularity result in
order to treat the class $2$ boundary condition  in \eqref{heat-intro} and the above $H^2$-spaces, cf.\ \eqref{X1-intro}.
Moreover, in addition to the sufficiency of a certain data regularity 
shown explicitly in \cite{LiMa72},  its necessity is important for the purpose of obtaining
an isomorphism, so Section~\ref{ivp-sect} accounts for this.

Section~\ref{fvp-sect} analyses final value problems in a
framework of Lax--Milgram operators $A$ that are $V$-coercive in a Gelfand triple
$V\hookrightarrow H\hookrightarrow V^*$. This extends the
$V$-elliptic case covered in \cite{ChJo18,ChJo18ax,JJ19}, as also done in \cite{JJ19cor}, but as a new feature
final value problems are treated in the more regular spaces 
$D(A)\hookrightarrow [D(A),H]_{1/2}\hookrightarrow H$. At this level isomorphic well-posedness is obtained too; cf.\
Corollary~\ref{fvA-cor}.

Section~\ref{Neu-sect} accounts for the treatment of the final value problem \eqref{heat-intro}
using the general results in Section~\ref{fvp-sect}; cf.\ Theorem~\ref{heatN-thm},
Theorem~\ref{heatN'-thm} and Corollary~\ref{heatN-cor} there. Some final remarks are gathered in Section~\ref{final-sect}.

\section{Preliminaries: Injectivity of Analytic Semigroups}
  \label{inj-sect}
For analysis of final value problems it is crucial that
analytic semigroups of operators always consist of \emph{injections}. This
shows up both technically and conceptually, that is, both in the proofs
and in the objects entering the theorems.

Some aspects of semigroup theory in a complex Banach space $B$ are therefore recalled.
Besides classical references by Davies~\cite{Dav80}, Pazy~\cite{Paz83}, Tanabe~\cite{Tan79} or
Yosida~\cite{Yos80}, more recent accounts are given by Engel and Nagel~\cite{EnNa00},  Arendt~\cite{Are04} or in \cite{ABHN11}.

The generator is
$\gen x=\lim_{t\to0^+}\frac1t(e^{t\gen}x-x)$, where $x$ belongs to the domain  $D(\gen)$ when the
limit exists. $\gen$ is a densely defined, closed linear 
operator in $B$ that for some $\omega \in\R$, $M \geq 1$ satisfies the resolvent estimates
$\|(\gen-\lambda I)^{-n}\|_{\B(B)}\le M/(\lambda-\omega)^n$ for $\lambda>\omega$, $n\in\N$.

The associated $C_0$-semigroup of operators $e^{t\gen}$ in $\B(B)$ is of type $(M,\omega)$: 
it fulfils $e^{t\gen}e^{s \gen}=e^{(s+t)\gen}$ for $s,t\ge0$, $e^{0\gen}=I$ (the identity) and 
$\lim_{t\to0^+}e^{t \gen}x=x$ for $x\in B$, whilst there is an estimate
\begin{align}  
  \|e^{t\gen}\|_{\B(B)} \leq M e^{\omega t} \quad \text{ for } 0 \leq t < \infty.
\end{align}
There is also a well-known translation trick, which is used repeatedly throughout, namely one has
\begin{equation}  \label{Amu-id}
  e^{t\gen}=e^{t\mu}e^{t(\gen-\mu I)} \qquad\text{for every $\mu\in\C$}.
\end{equation}
Indeed, the right-hand side is a $C_0$-semigroup
having $\gen$ as its generator (since $e^{t\mu}=1+t\mu+o(t)$), so the formula results from the
injectivity of $e^{t\gen}\mapsto \gen$. More explicitly, by the proof of the Hille--Yosida theorem, 
there is a bijection of the semigroups of type $(M,\omega)$ onto (the resolvents of) the stated class
of generators given by the Laplace transformation formula
\begin{equation}
  (\lambda I-\gen)^{-1}=\int_0^\infty e^{-t\lambda}e^{t\gen}\,dt  =\int_0^\infty e^{t(\gen-\lambda
    I)}\,dt ,\qquad \text{for $\Re\lambda>\omega$}.
\end{equation}
This formula also follows from the Fundamental Theorem for vector functions. Now, as the
evaluation map $\mathcal{E}_xT= Tx$ is bounded $\B(B)\to B$ for $x\in B$, the Bochner identity implies that for $\Re z>\omega$,
\begin{equation} \label{Eint'-id}
  \big(\int_0^\infty e^{t(\gen-zI)}\,dt\big)x= (zI-\gen)^{-1}x= \int_0^\infty e^{t(\gen-zI)}x\,dt.
\end{equation}

As for the \emph{injectivity}, 
recall that if $e^{t\gen}$ is analytic, then $u'=\gen u$, $u(0)=u_0$ is  for \emph{every} $u_0\in B$ uniquely solved by
$u(t)=e^{t\gen}u_0$. 
Here injectivity of $e^{t\gen}$ is  obviously equivalent to the geometric property that the trajectories of two solutions
$e^{t\gen}v$ and $e^{t\gen}w$ of $u'=\gen u$  have no confluence point in $B$ for $v\ne w$.

Despite this, the literature seems to have focused on examples of  non-invertible semigroups $e^{t\gen}$, like \cite[Ex.~2.2.1]{Paz83}; these necessarily concern non-analytic cases.
The well-known result below gives a criterion for $\gen$ to generate a
$C_0$-semigroup $e^{z\gen}$ that is defined and analytic for $z$ in the open sector
\begin{equation}
  S_{\theta}= \Set{z\in\C}{z\ne0,\ |\arg z | < \theta}.
\end{equation}
This notation is also used for the spectral sector in property (i) in the result:

\begin{proposition}  \label{Pazy'-prop} 
If $\gen$ generates a $C_0$-semigroup of type $(M,\omega)$ and $\omega\in\rho(\gen)$,
the following properties are equivalent for each
$\theta \in\,]0,\frac{\pi}{2}[\,$:
\begin{itemize}
  \item[{\rm (i)}]
  The resolvent set $\rho(\gen)$ contains $\{\omega\}\cup\big(\omega+S_{\theta+\pi/2}\big)$  and 
\begin{equation} 
  \sup\Set{|\lambda-\omega|\cdot\|(\lambda I - \gen)^{-1} \|_{\B(B)}}{\lambda\in\omega+S_{\theta+\pi/2}} <\infty. 
\end{equation}
  \item[{\rm (ii)}] 
 The semigroup $e^{t \gen}$ extends to an analytic semigroup $e^{z \gen}$ defined for $z\in
 S_{\theta}$ with
\begin{equation}
   \sup\Set{ e^{-z\omega}\|e^{z\gen}\|_{\B(B)}}{z\in \overline{S}_{\theta'}}<\infty \quad \text{whenever $0<\theta'<\theta$}. 
\end{equation}
\end{itemize}
In the affirmative case, $e^{t \gen}$ is differentiable in $\B(B)$  with $(e^{t\gen})' = \gen
e^{t\gen}$  for $t>0$, and 
\begin{align} \label{eta-est}
  \sup_{t>0} e^{-t\eta}\|e^{t\gen}\|_{\B(B)}+\sup_{t>0} te^{-t\eta}\|\gen e^{t\gen}\|_{\B(B)} <\infty
\end{align}
for  every $\eta>\alpha(\gen)$, whereby $\alpha(\gen)=\sup\Re\sigma(\gen)$
denotes the spectral abscissa of $\gen$.
\end{proposition}

If $\omega=0$,
the equivalence (i)$\iff$(ii) is contained in Theorem~2.5.2 in \cite{Paz83}.
This extends to $\omega>0$, using for both
implications that  \eqref{Amu-id}  holds with $\mu=\omega$ for complex $t$ in $S_\theta$ by unique analytic extension.

The first part of \eqref{eta-est} holds since analyticity implies $\alpha(\gen)=\omega_0$,
where the growth bound $\omega_0$ is 
the infimum of the $\omega$ such that $\|e^{t\gen}\|_{\B(B)}\le M e^{t\omega}$ for some $M$ 
(so $\eta=\omega$ is possible); cf.\ \cite[Cor.~IV.3.12]{EnNa00}.
For the last part we have $\omega>\alpha(\gen)=\omega_0$ (as $\omega\in\rho(\gen)$) and may hence consider 
$\alpha(\gen)<\eta'<\eta$, insert
$\gen=\eta' I+(\gen-\eta' I)$ and invoke the classical uniform bound of 
$t\|\gen e^{t\gen}\|_{\B(B)}$ from the case $\omega=0$.

The purpose of stating Proposition~\ref{Pazy'-prop} for general type $(M,\omega)$
semigroups is to emphasize that cases with $\omega>0$ only have other estimates in
the closed subsectors $\overline{S}_{\theta'}$, whereas the mere analyticity in $S_{\theta}$ 
is \emph{unaffected} by the translation by $\omega I$. This lead to the following sharpening of
\cite[Prop.\ 1]{ChJo18ax}:

\begin{proposition}[\cite{JJ19cor}]  \label{inj-prop}
If a $C_0$-semigroup $e^{t\gen}$ of type $(M,\omega)$ on a complex Banach space $B$ has an
analytic extension $e^{z\gen}$ to
$S_{\theta}$ for some $\theta>0$, then $e^{z\gen}$ is an \emph{injective} operator for every $z \in S_\theta$.
\end{proposition}

\begin{proof}
Let $e^{z_0 \gen} u_0 = 0$ hold for $u_0 \in B$ and $z_0 \in S_{\theta}$.
By the differential calculus in Banach spaces,
analyticity of $e^{z\gen}$ in $S_{\theta}$ carries over $f(z)= e^{z\gen}u_0$.
So for $z$ in a suitable open ball $B(z_0,r)\subset S_{\theta}$, a Taylor expansion and the identity
$f^{(n)}(z_0) = \gen^n e^{z_0 \gen}u_0$ for analytic semigroups (cf.~\cite[Lem.~2.4.2]{Paz83})  give
\begin{align}
  f(z) = \sum_{n=0}^{\infty} \frac{1}{n!}(z-z_0)^n f^{(n)}(z_0)=\sum_{n=0}^{\infty}
  \frac{1}{n!}(z-z_0)^n 
    \gen^n e^{z_0 \gen}u_0\equiv 0.
\end{align}
Therefore $f\equiv 0$ holds on $S_{\theta}$ by unique analytic extension, hence
$u_0 = \lim_{t \rightarrow 0^+} e^{t\gen} u_0 = \lim_{t \rightarrow 0^+}f(t) = 0$.
Thus the null space of $e^{z_0\gen}$ is trivial for every $z_0\in S_\theta$.
\end{proof}

As a corollary to the proof, in case $B=L_p(\Omega)$ for $1\le p<\infty$ and some
open set $\Omega\subset\Rn$, it is seen that if $u=e^{t \gen} u_0$ 
fulfils $u(t_0,\cdot)=0$ in an open subset $\Omega_0\subset\Omega$ for a given 
$t_0>0$,  the partial sums of the above power series converge to $u$ in $L_p(\Omega)$ for
$z=t$, $z_0=t_0$ with $|t-t_0|<t_0\tan\theta$; so  for each such $t$, a subsequence converges pointwise to $u(t,\cdot)$ a.e.\ in $\Omega_0$; which  simplifies to $0$
a.e.\ in $\Omega_0$ if $\gen$ preserves support in $\Omega$, so that $\gen^n u(t,\cdot)=0$ in $\Omega_0$. As an iteration will cover all $t>0$, one has the local result:

\begin{proposition} \label{locinj-prop}
  If, in addition to the hypothesis in Proposition~\ref{inj-prop}, the Banach space is given as $B=L_p(\Omega)$ for $1\le
  p<\infty$ and an open set $\Omega\subset\Rn$ in which $\gen$ preserves support, when a solution of  $u'=\gen u$ fulfils
  $u(t_0,\cdot)=0$ in an open subset $\Omega_0\subset\Omega$ for some $t_0>0$, then
  $u(t,x)=0$ for all $t>0$ and a.e.\ $x\in\Omega_0$.
\end{proposition}

\begin{remark}  \label{Yosida-rem}
Not surprisingly, Proposition~\ref{locinj-prop} has a forerunner in work of
Yosida \cite{Yos59}, who gave the above argument for $p=2$ under the extra assumption that 
$\gen$ is a strongly elliptic differential operator in $\Omega$.
The concise conclusion in Proposition~\ref{inj-prop} was not reached in \cite{Yos59} (although
$\Omega_0=\Omega$ is a possibility in Proposition~\ref{locinj-prop}), and it seems not
to have appeared in the semigroup literature during the following decades, until it was
shown for $\omega=0$ in \cite{ChJo18,ChJo18ax}.
Proposition~\ref{inj-prop} was anticipated  for  
$z>0$, $\theta\le \pi/4$ and $B$ a Hilbert space in \cite{Sho74},
but not quite obtained; cf.\ details in \cite[Rem.~1]{ChJo18ax} and  \cite[Rem.~3]{JJ18logconv}.
Masuda \cite{Mas67} used the unique continuation property to obtain the stronger result that $u=0$
extends from $\{t_0\}\times\Omega_0$ to $\R_+\times\Omega$; Rauch \cite[Cor.~4.3.9]{Rau91} gave 
a version for $\lap$ on $\Rn$.
\end{remark}

As a result of the above injectivity, for an \emph{analytic} semigroup 
$e^{t\gen}$ we may consider its inverse that, like when $e^{t\gen}$ forms a group in
$\mathbb{B}(B)$, may be denoted for $t>0$ by
$e^{-t\gen} = (e^{t\gen})^{-1}$.
Clearly $e^{-t\gen}$ maps its domain $D(e^{-t\gen})=R(e^{t\gen})$ bijectively onto $H$, and
it is  an unbounded, but closed operator in $B$. 

Specialising to a Hilbert space $B=H$, then also $(e^{t\gen})^*=e^{t\gen^*}$ is analytic, so 
$Z(e^{t\gen^*})=\{0\}$ holds for its null space by Proposition~\ref{inj-prop}; whence
$D(e^{-t\gen})$ is dense in $H$.
Some further basic properties are:

\begin{proposition}{\cite[Prop.\;2]{ChJo18ax}}  \label{inverse-prop}
The above inverses $e^{-t\gen}$ form a semigroup of unbounded operators in $H$,
\begin{equation} 
  e^{-s\gen}e^{-t\gen}= e^{-(s+t)\gen} \qquad \text{for $t, s\ge0$}.
\end{equation}
This extends to $(s,t)\in\R\times \,]-\infty,0]$, whereby $e^{-(t+s)\gen}$  may be unbounded for $t+s>0$. 
Moreover, as unbounded operators the $e^{-t\gen}$ commute with $e^{s \gen}\in \B(H)$, that is,
$e^{s \gen}e^{-t\gen}\subset e^{-t\gen}e^{s\gen}$ for $t,s\ge0$,
and have a descending chain of domains,
$H\supset  D(e^{-t\gen}) \supset D(e^{-t'\gen})$ for $0<t<t'$.
\end{proposition}

\begin{remark}
  The domains $D(e^{-t\gen})$ of the inverses have been introduced independently in the literature
  on the regularisation of final value problems (albeit for $t>T$). A very recent example is \cite{Fur19}.
\end{remark}

\section{Evolution equations revisited}
  \label{ivp-sect}
This section outlines the prerequisites on evolution equations needed in Sections~\ref{fvp-sect}--\ref{Neu-sect}.
The material is entirely classical, and the treatise of Lions and Magenes  \cite{LiMa72} is chosen as the main source, although the below Theorem~\ref{HDA-thm} was not stated there (it is known nowadays in a more general setting).
\bigskip

The basic analysis is made for a
Lax--Milgram operator $A$ defined in $H$ from a $V$-coercive sesquilinear form $a(\cdot,\cdot)$ 
in a Gelfand triple, i.e.,
three separable, densely injected Hilbert spaces $V\hookrightarrow H\hookrightarrow V^*$ 
having the norms $\|\cdot\|$, $|\cdot|$ and $\|\cdot\|_*$, respectively. Hereby $V$ is the form
domain of $a$; and $V^*$ the antidual  of $V$.  
Specifically there are constants $C_j>0$ and $k\in\R$ 
such that all $u, v\in V$ satisfy $\| v\|_*\le C_1|v|\le C_2 \| v\|$ and 
\begin{equation} \label{coerciv-id}
  |a(u,v)|\le C_3\|u\|\,\|v\|\,\qquad \Re a(v,v)\ge C_4\|v\|^2-k|u|^2. 
\end{equation}
In fact, $D(A)$ consists of the $u\in V$ for
which $a(u,v)=\scal{f}{v}$   for some $f\in H$ and all $v\in V$; then $Au=f$.
Hereby $\scal{u}{v}$ denotes the inner product in $H$. 
There is also an extension $A\in\B(V,V^*)$ defined by the identity $\dual{Au}{v}=a(u,v)$ for all
$u,v\in V$. This is uniquely determined as $D(A)$ is dense in $V$.

Both $a$ and $A$ are referred to as $V$-elliptic if the above holds for $k=0$; then 
$A\in\B(V,V^*)$ is a bijection.
E.g.\ \cite{G09}, \cite{Hel13} or \cite{ChJo18ax} give more
details on the set-up and basic properties of the unbounded, but closed operator $A$ in $H$.
Throughout $D(A)$ is endowed with the graph norm, which is complete.

The operator $A$ is self-adjoint  if and only if $a(v,w)=\overline{a(w,v)}$, which is \emph{not} assumed.
 $A$ may also be nonnormal in general. For a non-trivial example based on the
advection-diffusion operators $-\partial_x^2\pm \partial_x$ with mixed Dirichlet, Neumann and Robin
conditions on an interval $\,]\alpha,\beta[\,$, one may consult \cite[Ex.\ 1]{JJ18logconv},
where both $D(A^*)\setminus D(A)\ne\emptyset$ and $D(A)\setminus D(A^*)\ne\emptyset$ are shown to hold.

\bigskip

In this set-up, the general Cauchy problem is, for given data  
$f\in L_2(0,T; V^*)$ and $u_0\in H$,
to determine the  $u\in{\cal D}'(0,T;V)$ (i.e.\ the space of continuous linear maps 
$C_0^\infty(]0,T[)\to V$, cf.\ \cite{Swz66}) satisfying
\begin{equation}
  \label{ivA-intro}
  \left.
  \begin{aligned}
  \partial_tu +Au &=f  &&\quad \text{in ${\cal D}'(0,T;V^*)$},
\\
  u(0)&=u_0 &&\quad\text{in $H$}.
\end{aligned}
\right\}
\end{equation}
By definition of Schwartz' vector distribution space ${\cal D}'(0,T;V^*)$, the first equation above means that 
for every scalar test function $\varphi\in C_0^\infty(]0,T[)$
the identity $\dual{u}{-\varphi'}+\dual{A u}{\varphi}=\dual{f}{\varphi}$ holds in $V^*$.

As is well known, a wealth of parabolic Cauchy problems with homogeneous boundary 
conditions have been treated via triples $(H,V,a)$ and the ${\cal D}'(0,T;V^*)$
framework in \eqref{ivA-intro}; cf.\ the
work of Lions and Magenes~\cite{LiMa72}, Tanabe~\cite{Tan79}, Temam~\cite{Tem84}, Amann
\cite{Ama95} etc. 

For the problem \eqref{ivA-intro}, it is classical to seek solutions $u$ in the Banach space 
\begin{equation}
  \begin{split}
  X=&L_2(0,T;V)\bigcap C([0,T];H) \bigcap H^1(0,T;V^*),
\\
  \|u\|_X=&\big(\int_0^T \|u(t)\|^2\,dt+\sup_{0\le t\le T}|u(t)|^2+\int_0^T (\|u(t)\|_*^2   +\|u'(t)\|_*^2)\,dt\big)^{1/2}.   
  \end{split}
  \label{eq:X}
\end{equation}
However, to point out a redundancy, note that 
the  Banach space $X$ can have its norm rewritten\,---\,using the Sobolev space 
$H^1(0,T;V^*)=\Set{u\in L_2(0,T;V^*)}{\partial_t u\in L_2(0,T;V^*)}$\,---\, in  the form
\begin{align}  \label{eq:Xnorm}
  \|u\|_{X} = \big(\|u\|^2_{L_2(0,T;V)} + \sup_{0 \leq t \leq T}|u(t)|^2 + \|u\|^2_{H^1(0,T;V^*)}\big)^{1/2},
\end{align}
Here there is a well-known inclusion $L_2(0,T;V)\cap H^1(0,T;V^*)\subset C([0,T];H)$ and an associated
Sobolev inequality for vector functions
$
  \sup_{0\le t\le T}| u(t)|^2\le (1+\frac{C_2^2}{C_1^2T})\int_0^T \|u(t)\|^2\,dt+\int_0^T \|u'(t)\|_*^2\,dt
$ (cf.\ \cite{ChJo18ax}).
Hence one can safely omit the space $C([0,T];H)$ in  \eqref{eq:X} and remove
$\sup_{[0,T]}|\cdot|$ from $\| \cdot\|_{X}$. Similarly $\int_0^T\|u(t)\|_*^2\,dt$ is redundant
in \eqref{eq:X} because $\|\cdot\|_*\le C_2\|\cdot\|$, so an equivalent norm on $X$ is given by
\begin{equation} \label{Xnorm-id}
  \vvvert u\vvvert_X =
  \big(\int_0^T \|u(t)\|^2\,dt +\int_0^T \|u'(t)\|_*^2\,dt\big)^{1/2}.   
\end{equation}
Thus $X$ is more precisely a Hilbertable space, as $V^*$ is so.
But the form given in \eqref{eq:X} serves the purpose of emphasizing the properties of the solutions.

For \eqref{ivA-intro} the following result is known from the work of Lions and Magenes \cite{LiMa72}:
 
\begin{proposition}  \label{LiMa-prop}
Let $V$ be a separable Hilbert space with $V \subset H$ algebraically, topologically and densely,
and let $A$ denote the Lax--Milgram operator induced by a $V$-coercive, bounded 
sesquilinear form $a(\cdot,\cdot)$ on $V$, as well as its extension $A\in\B(V,V^*)$.
To given $u_0 \in H$ and $f \in L_2(0,T; V^*)$ there exists
a uniquely determined solution $u$ belonging to $X$, cf.\ \eqref{eq:X},
of the Cauchy problem \eqref{ivA-intro}.

The solution operator $\mathcal{R}\colon (f,u_0)\mapsto u$ is bounded $L_2(0,T; V^*)\oplus H\to X$, 
and problem \eqref{ivA-intro} is well posed.
\end{proposition}

The existence and uniqueness statements in Proposition~\ref{LiMa-prop}  
were mentioned after \cite[Sect.~3.4.4]{LiMa72},
where they indicated an extension to the coercive case  by means of a translation trick as  in \eqref{Amu-id}.
(Details may e.g.\ be found in \cite{JJ19cor} together with an explicit proof of boundedness of $\mathcal{R}$.)

As a note on the equation $u'+Au=f$ with $u\in X$, the continuous function
$u\colon[0,T]\to H$ fulfils $u(t)\in V$ for a.e.\ $t\in\,]0,T[\,$, so the extension $A\in
\B(V,V^*)$ applies for a.e.\ $t$. Hence $Au(t)$ belongs to $L_2(0,T;V^*)$.

\begin{remark} \label{generator-rem}
It is recalled that $\gen=-A$ generates an analytic semigroup $e^{-zA}$ in both
$\B(H)$ and $\B(V^*)$, each defined in $S_{\theta}$ for $\theta=\operatorname{arccot}(C_3/C_4)>0$, cf.\ \eqref{coerciv-id}.  
For $V$-elliptic $A$, these known generation results were concisely shown in \cite[Lem.\ 4]{ChJo18ax} e.g.
In the $V$-coercive case, this applies to $\gen=-(A+kI)$, so \eqref{Amu-id} gives
$e^{-zA}=e^{kz}e^{-z(A+kI)}$ for $z\ge0$; which then
defines $e^{-zA}$ by the right-hand side for every $z\in S_\theta$.
(An involved argument was given in \cite[Thm.\ 7.2.7]{Paz83} for uniformly
strongly elliptic differential operators.)
\end{remark}

In addition to the existence and uniqueness statements in Proposition~\ref{LiMa-prop}, 
it is useful to have an expression for the solution $u\in X$.
In view of Remark~\ref{generator-rem}, the canonical candidate is the Duhamel fomula,
\begin{equation} \label{u-id}
  u(t) = e^{-tA}u_0 + \int_0^t e^{-(t-s)A}f(s) \,ds \qquad\text{for } 0\leq t\leq T.
\end{equation}
This is known from analytic semigroup theory to \emph{produce} a classical solution,
 cf.\ the books of Pazy~\cite[Cor.\ 4.3.3]{Paz83}, Amann~\cite[Rem.~2.1.2]{Ama95} or Arendt~\cite{Are04},
provided that the $H$-valued source term $f(t)$  is  H{\"o}lder continuous of some order  $\sigma\in\,]0,1[\,$,
\begin{equation} \label{Holder-id}
  \sup\Set{|f(t)-f(s)|\cdot|t-s|^{-\sigma}}{0\le s<t\le T}<\infty.
\end{equation}
This was exploited for the Hadamard well-posedness of \eqref{heat-intro} shown in \cite{JJ19cor}. 

But in the present framework, Duhamel's formula \eqref{u-id} needs another justification,
because the $u\in X$ are solutions in the distribution sense, and the data $f\in L_2(0,T;V^*)$ need  not  be H\"older continuous.

However,  by virtue of Remark~\ref{generator-rem} and Proposition~\ref{inj-prop}, the operators $e^{-tA}$ are always injective, even when $A$ is merely $V$-coercive, and this allows an easy extension of the
classical integration factor technique, yielding a proof of  \eqref{u-id} for all $u_0\in H$ and $f\in L_2(0,T;V^*)$; cf.\ \cite{JJ19cor}. Thus one has

\begin{proposition}   \label{Duhamel-prop}
The unique solution $u$ in $X$ provided by Proposition~\ref{LiMa-prop} is given by Duhamel's
formula \eqref{u-id}, where each of the three terms belongs to $X$.
\end{proposition}

For the treatment of the heat equation further below, it is decisive to know that $u(t)\in D(A)$ in
order to make sense of the Neumann boundary condition. The well-known way to obtain this is to work with the more regular solutions that belong to the subspace $X_1\subset X$ given by
\begin{align}
  X_1 &= L_2(0,T;D(A)) \bigcap  H^1(0,T; H),
   \label{X1-eq}
\\
  \|u\|_{X_1}&= \big(\int_0^T\|u(t)\|^2_{D(A)}\,dt 
     +
                 \int_0^T|\partial_t u(t)|^2)\,dt  \Big)^{1/2}.
\end{align}
Here it is classical  that the intersection $L_2(0,T;D(A))\cap H^1(0,T;H)$ is
embedded into $C([0,T];U)$ for the interpolation space $U=[D(A),H]_{1/2}$; cf.\ Theorem 3.1 in \cite{LiMa72}.
 For the norm on this space, the reader is referred to the formula given in  \cite{LiMa72}, which exploits the spectral decomposition of self-adjoint operators in terms of direct Hilbert integrals. 

Along with the said inclusion, there is the estimate
\begin{equation} \label{CDAH-ineq}
  \sup_{t\in[0,T]} \|u(t)\|_{[D(A),H]_{1/2}}\le 
  c\big (\|u\|_{L_2(0,T;D(A))}+\|u\|_{H^1(0,T;H)}\big),
\end{equation}
which yields an equivalent norm on $X_1$ when added to $\|u\|_{X_1}$.
Thus one obtains the well-known  additional regularity property of solutions in $X_1$,
\begin{equation} \label{CDAH-id}
  u\in  C\big([0,T];[D(A),H]_{1/2}\big).
\end{equation}
Furthermore, the initial data $u_0$ must therefore be given in $[D(A),H]_{1/2}$.
That this space is compatible with the requirements that $u\in X_1$ is seen from (i)$\iff$(ii)  by taking $X=D(A)$, $Y=H$, $\theta=1/2$ and $a=u_0$ in the clarifying result  \cite[Thm.~I.10.1]{LiMa72}\,---\,although this had an unfortunate lack of quantifiers that may be remedied thus:

\begin{theorem}[\cite{LiMa72}] \label{LM101-thm}
  Let $X$, $Y$ be arbitrary complex Hilbert spaces with a continuous dense injection $X\hookrightarrow Y$. Then the
  following properties are equivalent for each $\theta\in\,]0,1[\,$ and $a\in Y$:
  \begin{alignat*}{2}
    \text{\upn{(i)}}\qquad&& &a\in [X,Y]_\theta;
\\
    \text{\upn{(ii)}}\qquad&& & a=u(0) \ \text{ for some $u\colon\R_+\to Y$ such that } t^{\theta-\frac12} u\in
    L_2(\R_+;X),\
    t^{\theta-\frac12} u'\in L_2(\R_+;Y);
\\
    \text{\upn{(iii)}}\qquad&& &t^{\theta-\frac12}\frac{e^{t\gen}a-a}t\in L_2(\R_+;Y)\ \text{ for 
    some generator $\gen$ of a $C_0$-semigroup in $Y$,}
\\
  && &\text{such that $D(\gen)=X$ holds with equivalent norms.}
  \end{alignat*}
In the affirmative case, \upn{(iii)} is valid for \emph{every} $C_0$-semigroup of the mentioned kind,
with equivalent norms
\begin{equation}  \label{norm-eqv}
  \|a\|_{[X,Y]_\theta}\qquad\text{and}\qquad \Big(\|a\|_Y^2+\int_0^\infty t^{2(\theta-\frac12)}\|\frac{e^{t\gen_0}a-a}t\|_Y^2\,dt\Big)^{1/2}
\end{equation}
whenever $\gen_0$ is a selfadjoint generator fulfilling \upn{(iii)} 
(these induce the  interpolation spaces $[X,Y]_\theta$).
\end{theorem}

Indeed, one may readily see that the proof given  in \cite{LiMa72} (where
$\alpha:=\theta-\frac12$) achieves via the Hardy--Littlewood--Polya inequality 
 that (ii)$\implies$(iii) for \emph{arbitrary} 
$C_0$-semigroups fulfilling the criteria in (iii).
Then the proof there shows that when $\gen$ satisfies (iii), then (ii) holds for a specific
function $u$. Finally equivalence of (i) and (iii) is verified if one takes for ${\gen_0}$
any of the generators that can be utilised in the definition of $[X,Y]_\theta$, whereby the equivalence
of the norms also is obtained.

Thus prepared, it is now possible to formulate the following  regularity result, which is well known among experts, but here given in a form suitable for  the purposes in Section~\ref{fvp-sect}--\ref{Neu-sect}:

\begin{theorem} \label{HDA-thm}
For the unique solution $u\in X$ of the Cauchy problem \eqref{ivA-intro} for given data $(f,u_0)\in Y$, the
following additional properties are equivalent:
\begin{align*}
  \text{\upn{(i)}}\qquad f&\in L_2(0,T;H),\qquad\quad\ \ \;  u_0\in [D(A),H]_{1/2}    
\\  \label{CHDA-eq}
  \text{\upn{(ii)}}\qquad u&\in L_2(0,T;D(A))\bigcap C([0,T];[D(A),H]_{1/2})\bigcap H^1(0,T;H).
\end{align*}
In the affirmative case all terms in Duhamel's formula \eqref{u-id} belong to the space in
\upn{(ii)}, and there is a constant $c>0$ such that each solution $u$ corresponding to data $(f,u_0)$
satisfying \upn{(i)} will fulfil
\begin{equation} \label{ivp-est}
  \int_0^T\big(|u(t)|^2+|Au(t)|^2+|u'(t)|^2\big)\,dt\le c
  \Big(\int_0^T|f(t)|^2\,dt +\|u_0\|_{[D(A),H]_{1/2}}^2\Big).
\end{equation}
If $A^*=A$ in $H$, the form domain $V$ identifies with
the space $[D(A),H]_{1/2}$ in \upn{(i)}, \upn{(ii)} and \eqref{ivp-est} above.
\end{theorem}

This result was not formulated by Lions and Magenes~\cite{LiMa72} but undoubtedly known to them. For a brief discussion,  note that (ii)$\implies$(i) is obvious from the mapping properties of $\cal{P}(u)=(u'+Au, u(T))$. As for (i)$\implies$(ii), it is clear from Theorem~\ref{LM101-thm} that it suffices to treat the case $u_0=0$ for arbitrary $f$ in $L_2(0,T;H)$; this was  done in \cite[Sect.~IV.3]{LiMa72} using the general theory of Laplace transformation of vector distributions and convolutions $\cal{G}*(A+\frac\partial{\partial t})$ exposed in \cite{Swz52} or \cite[Ch.~VIII]{Swz66}. (One can also make a version based only on the semigroup $e^{-tA}$.) Consequently also the last term $e^{-tA}u_0$ in \eqref{u-id} belongs to the space $X_1$ in (ii). The estimate is seen from the Closed Graph Theorem, since $\cal R$ is 
the inverse of $\cal P$, which is bounded from $X_1$ to $L_2(0,T;H)\oplus [D(A),H]_{1/2}$.

The final fact that $V=[D(A),H]_{1/2}$ holds is also standard when $A^*=A\ge0$ (both spaces equal $D(A^{1/2})$ then, by spectral theory), and this applies in the general self-adjoint $V$-coercive case to $A+k'I$ for $k'>k$, which has the same domain and form domain as $A$.

\begin{remark}  \label{maxreg-rem}
  The estimate in Theorem~\ref{HDA-thm} has been known for decades by experts in evolution equations.
More generally, with $H$ replaced by a UMD Banach space $B$, and $H^1(0,T; H)$ replaced by the $L_p$-Sobolev space $W^1_p(0,T;B)$, $1<p<\infty$, it has been a major theme, known as maximal regularity, since the 1980's to establish that $u\mapsto u'+Au$  as a map $W^1_p(0,T;B)\cap L_p(0,T;D(A))\to L_p(0,T;B)$ has a bounded inverse. This was first shown by Dore and Venni~\cite{DoVe87} under suitable assumptions, but this generality is not needed here
(though for $p=2$ the choices $B=V^*$ and $B=H$  give back Proposition~\ref{LiMa-prop} and Theorem~\ref{HDA-thm} with $u_0=0$).
The reader may consult \cite{Ama95, DHP03}, or  for a survey of maximal regularity also \cite[Sec.~5]{Are04}.
\end{remark}

To complete this review of linear evolution equations,
it is natural to introduce a notation for the full \emph{yield} of the source term
$f\colon \,]0,T[\, \to V^*$, namely the following vector that a priori belongs to $V^*$
\begin{equation} \label{yf-eq}
  y_f= \int_0^T e^{-(T-t)A}f(t)\,dt.
\end{equation}
In fact $y_f\in H$ as by Duhamel's formula it equals 
the final state of a solution in $C([0,T],H)$ of a Cauchy problem having $u_0=0$. 
Moreover, Theorem~\ref{HDA-thm} shows that $y_f\in [D(A),H]_{1/2}$ whenever $f\in L_2(0,T;H)$.

For $t=T$ formula \eqref{u-id} now obviously yields a 
\emph{bijection} $u(0)\longleftrightarrow u(T)$ between the initial
and terminal states (for fixed $f$), as one can solve for $u_0$ by means of the inverse $e^{TA}$. 
Indeed, all terms in \eqref{u-id} belong to $C([0,T];H)$, so evaluation at $t=T$
gives $u(T)=e^{-TA}u(0)+y_f$; cf.\ \eqref{yf-eq}. This is a flow map 
\begin{equation} \label{flow-id}
  u(0)\mapsto u(T).  
\end{equation}
Invoking injectivity of $e^{-TA}$ once again, and that \eqref{u-id} implies
$u(T)-y_f=e^{-TA}u(0)$, which clearly belongs to $D(e^{TA})$, the flow is inverted  by
\begin{equation}  \label{u0uT-id}
  u(0)=e^{TA}(u(T)-y_f).  
\end{equation}
In other words, not only are the solutions in $X$ to $u'+Au=f$ parametrised by the initial
states $u(0)$ in $H$ (for fixed $f$) according to Proposition~\ref{LiMa-prop}, but also the final
states $u(T)$ are parametrised by the $u(0)$. 

For one thing, this means that the differential equation $u'+Au=f$ has the 
\emph{backward uniqueness} property regardless of whether $A$ itself is injective or not:
that is, $u(t)=0$ holds in $H$ for all $t\in[0,T[\,$ if $u(T)=0$.
This property has been studied for decades in various situations, cf.\ Remark~\ref{Yosida-rem} and
Remark~\ref{LM-rem} below.

Secondly, the remarks on the above flow lead to the isomorphic well-posedness in the next section.

\section{Final value problems with coercive generators}
\label{fvp-sect}

In the framework of Section~\ref{ivp-sect}, the general final value problem is, for given data  
$f\in L_2(0,T; V^*)$ and $u_T\in H$,
to determine the  $u\in{\cal D}'(0,T;V)$ satisfying
\begin{equation}
  \label{fvA-intro}
  \left.
  \begin{aligned}
  \partial_tu +Au &=f  &&\quad \text{in ${\cal D}'(0,T;V^*)$},
\\
  u(T)&=u_T &&\quad\text{in $H$}.
\end{aligned}
\right\}
\end{equation}
The point of departure for this is to make a comparison of \eqref{fvA-intro} with the corresponding Cauchy
problem for the equation $u'+Au=f$, cf.\ \eqref{ivA-intro}.
Thus it would be natural to seek solutions $u$ in the same space $X$ in \eqref{eq:X}.
As shown first for the $V$-elliptic case in \cite{ChJo18ax}, this is possible only for data 
$(f, u_T)$ subjected to certain \emph{compatibility conditions}, which have a special form for final value problems.

The compatibility condition is formulated by means of the inverse $e^{tA}$ that  enters the theory through its
domain $D(e^{tA})$, to which Proposition~\ref{inverse-prop} applies. Although this identifies with
the range $R(e^{-tA})$ in the algebraic sense, it has the virtue of being a Hilbert space under 
the graph norm $\|u\|=(|u|^2+|e^{tA}u|^2)^{1/2}$.

The remarks on $y_f$ made after \eqref{yf-eq} make it clear that in the following
general result the difference in \eqref{eq:cc-intro} is a priori a member of  $H$.
The theorem relaxes the assumption of $V$-ellipticity in \cite{ChJo18,ChJo18ax} to $V$-coercivity.
Because of its relative novelty,
it is given here with details for the reader's sake.

\begin{theorem}[\cite{JJ19cor}] \label{fvp-thm}
  Let $A$ be a $V$-\emph{coercive} Lax--Milgram operator defined from a triple $(H,V,a)$ as above.
  The abstract final value problem \eqref{fvA-intro} then has a solution $u(t)$ belonging to the space
  $X$ in \eqref{eq:X} if, and only if, the data $(f,u_T)$ belong to the Banach space $Y$, which is
  the subspace 
  \begin{equation}
    Y\subset L_2(0,T; V^*)\oplus H
  \end{equation}
  defined by the compatibility condition  
  \begin{equation}
    \label{eq:cc-intro}
    u_T-\int_0^T e^{-(T-t)A}f(t)\,dt \ \in\  D(e^{TA}).
  \end{equation}  
In the affirmative case, the solution $u$ is uniquely determined in $X$ and 
\begin{equation}
  \label{eq:Y-intro}
      \|u\|_{X} \le
  c
  \Big(|u_T|^2+\int_0^T\|f(t)\|_*^2\,dt+\Big|e^{TA}\big(u_T-\int_0^Te^{-(T-t)A}f(t)\,dt\big)\Big|^2\Big)^{1/2}
  =:c \|(f,u_T)\|_Y,
\end{equation}
whence the solution operator $(f,u_T)\mapsto u$ is continuous $Y\to X$. Moreover,
\begin{equation} \label{eq:fvp_solution}
  u(t) = e^{-tA}e^{TA}\Big(u_T-\int_0^T e^{-(T-t)A}f(t)\,dt\Big) + \int_0^t e^{-(t-s)A}f(s) \,ds,
\end{equation}
where all terms belong to $X$ as functions of $t\in[0,T]$, and the difference in
\eqref{eq:cc-intro} equals $e^{-TA}u(0)$ in $H$.
\end{theorem}

\begin{proof}
If \eqref{fvA-intro} is solved by $u \in X$, then $u(T)=u_T$ is reached from
the unique initial state $u(0)$ in \eqref{u0uT-id}. But the argument  for \eqref{u0uT-id}
showed that $u_T-y_f = e^{-TA} u(0)\in D(e^{TA})$, 
so \eqref{eq:cc-intro} is necessary. 

Given data $(f,u_T)$  fulfilling \eqref{eq:cc-intro},
then $u_0 = e^{TA}(u_T - y_f)$ is a well-defined vector in $H$, so 
Proposition~\ref{LiMa-prop} yields a
function $u\in X$ solving $u' +Au = f$ and $u(0)=u_0$. 
By \eqref{flow-id},
this $u(t)$ clearly has final state $u(T)=e^{-TA}e^{TA}(u_T-y_f)+y_f=u_T$, hence satisfies both equations in \eqref{fvA-intro}.
Thus \eqref{eq:cc-intro} suffices for solvability.

In the affirmative case, \eqref{eq:fvp_solution} results for any solution $u\in X$ by inserting
formula \eqref{u0uT-id} for $u(0)$ into \eqref{u-id}. 
Uniqueness of $u$ in $X$ is seen from the  right-hand side of \eqref{eq:fvp_solution}, where all
terms depend only on the given $f$, $u_T$, $A$ and $T>0$. 
That each term in \eqref{eq:fvp_solution}
is a function belonging to $X$ was seen in Proposition~\ref{Duhamel-prop}.

Moreover, the solution can be estimated in $X$ by substituting the expression \eqref{u0uT-id} for $u_0$
into the inequality that expresses the boundedness of $\mathcal{R}$ in Proposition~\ref{LiMa-prop},
\begin{equation}
 \vvvert u\vvvert_X^2 \leq c\big(|u_0|^2 + \int_0^T\|f(s)\|_{*}^2\,ds\big)
 \le c(|e^{TA}(u_T-y_f)|^2+\|f\|_{L_2(0,T;V^*)}^2).
\end{equation}
Here one may add $|u_T|^2$ on the right-hand side to arrive at the expression for $\|(f,u_T)\|_Y^2$  in \eqref{eq:Y-intro}.
\end{proof}

\begin{remark} \label{P-rem}
  It is clear from the definitions and proofs that $\cal{P}u=(u'+ Au, u(T))$ is 
bounded $X\to Y$.
The statement in Theorem~\ref{fvp-thm} means that the solution operator $\cal{R}(f,u_T)=u$
is well defined, bounded  and satisfies $\cal{P}\cal{R}=I$; but by the uniqueness also $\cal{R}\cal{P}=I$
holds. Hence $\cal{R}$ is a linear homeomorphism $Y\to X$.
\end{remark}

The norm on the data space $Y$ in \eqref{eq:Y-intro} is seen at once to be the graph norm of the
composite map 
\begin{equation}
   (f,u_T)\ \mapsto u_T-y_f\ \mapsto \ e^{TA}(u_T-y_f)
\end{equation}
that  in terms
of the first part $\Phi(f,u_T)=u_T-y_f$ is the operator
\begin{equation} \label{EFi-id}
  L_2(0,T; V^*)\oplus H \xrightarrow[\;]{\quad \Phi\quad} H\xrightarrow[\;]{\quad e^{TA}\quad} H.
\end{equation}
In fact,  the solvability criterion \eqref{eq:cc-intro} is met if and only if $e^{TA}\Phi$ is defined at
$(f,u_T)$, so the data space $Y$ is its domain. Being an inverse, $e^{TA}$ is a closed operator in $H$; 
hence $e^{TA}\Phi$ is closed, and $Y=D(e^{TA}\Phi)$ is complete. Now, since in \eqref{eq:Y-intro} the Banach space $V^*$
is Hilbertable, so is $Y$.

In control theoretic terms, the role of $e^{TA}\Phi$ is also for $V$-coercive $A$ to provide the unique initial state 
$u(0)=e^{TA}\Phi(f,u_T)=e^{TA}(u_T-y_f)$,
which is steered by $f$ to the final state $u(T)=u_T$ at time $T$;
cf.\ the Duhamel formula \eqref{u-id}.

Criterion \eqref{eq:cc-intro} is a generalised \emph{compatibility} condition on
the data $(f,u_T)$; such conditions have long been known in the theory of parabolic problems, cf.\
Remark~\ref{GS-rem}.
The presence of $e^{-(T-t)A}$ and the integral over $[0,T]$ makes
\eqref{eq:cc-intro} \emph{non-local} in both space and time. This 
aspect is further complicated by the reference to the abstract domain $D(e^{TA})$, which for larger
final times $T$ typically gives increasingly stricter conditions:

\begin{proposition}
If the spectrum $\sigma(A)$ of $A$ is not contained in the strip
$\set{z\in\C}{-k\le \Re z\le k}$, whereby $k$ is the constant from \eqref{coerciv-id}, then 
the domains $D(e^{tA})$ form a strictly descending chain, that is,
\begin{equation} \label{dom-intro}
 H\supsetneq D(e^{tA})\supsetneq D(e^{t' A}) \qquad\text{ for  $0<t<t'$}.
\end{equation}
\end{proposition}

This results from the injectivity of $e^{-tA}$ via known facts for semigroups 
reviewed in \cite[Thm.\ 11]{ChJo18ax} (with reference to \cite{Paz83}), and the arguments
given for $k=0$ in \cite[Prop.\ 11]{ChJo18ax} apply mutatis mutandis.

\bigskip

The regularity result in Theorem~\ref{HDA-thm} gives rise to a well-posedness result further below,
which concerns some more regular data and solution spaces.

For convenience we shall for an unbounded operator $S\colon X\to Y$ between two general Banach spaces $X$,
$Y$ and any given subspace $U\subset Y$ adopt the notation
\begin{equation}  \label{DSU-id}
  D(S;U)=\Set{x\in D(S)}{Sx\in U} = D(S)\bigcap S^{-1}(U).
\end{equation}
This is the domain of the composite map $I_US\colon D(S)\to Y$ whereby $I_U$ denotes the inclusion map $U\to Y$.
When $S$ has closed graph and $I_U$ is bounded with respect to  some complete norm $\|\cdot\|_U$ on $U$, then 
the domain $D(S;U)$ is complete with respect to the modified graph norm $\|x\|_{D(S;U)}=\|x\|_X+\|Sx\|_U$.

This applies especially to the inverse operator $e^{TA}$, for which $D(e^{TA};U)=
e^{-TA}(U)$ when $U\subset H$.

As a more regular solution space for \eqref{fvA-intro} one may use \eqref{X1-eq} ff.,
\begin{align}
  X_1 &= L_2(0,T;D(A)) \bigcap C([0,T]; [D(A),H]_{1/2}) \bigcap H^1(0,T; H),
   \label{X1-id}
\\
  \|u\|_{X_1}&= \big(\int_0^T\|u(t)\|^2_{D(A)}\,dt 
                 +\sup_{t\in[0,T]}\|u(t)\|_{[D(A),H]_{1/2}}^2+
                 \int_0^T|\partial_t u(t)|^2)\,dt  \Big)^{1/2}.
\end{align}
The corresponding data space $Y_1$ is given  as
\begin{align}
  Y_1&= \left\{ (f,u_T) \in L_2(0,T;H) \oplus [D(A),H]_{1/2} \Bigm|  
                  u_T - y_f \in D(e^{TA}; [D(A),H]_{1/2}) \right\},
\label{Y1-id}
\\
 \| (f,u_T) \|_{Y_1}  
  &= \Big(\int_0^T|f(t)|^2\,dt 
  + \|u_T\|_{[D(A),H]_{1/2}}^2+ \| e^{TA}(u_T - y_f )\|_{[D(A),H]_{1/2}}^2\Big)^{1/2}.
\end{align} 
It was noted above that the yield of the source term $y_f=\int_0^T e^{-(T-t)A}f(t)\,dt$  a priori belongs to 
$[D(A),H]_{1/2}$ for the stipulated $f$ in $L_2(0,T;H)$.

It is an exercise to show that $Y_1$ is a Banach space, for if $(f_n,u_{T,n})$ is a Cauchy sequence
in $Y_1$, then $f_n$, $u_{T,n}$ and $e^{TA}(u_{T,n}-y_{f_n})$ converge to some $f$ in $L_2(0,T;H)$
and $u_T$, $v$ in $[D(A),H]_{1/2}$, respectively; for reasons of continuity, $y_{f_n}\to y_f$ so that 
$u_{T,n}-y_{f_n}\to u_T-y_f$ in $H$ for $n\to\infty$; as $e^{TA}$ is closed in $H$, it follows that
$u_T-y_f$ belongs to $D(e^{TA})$ with $e^{TA}(u_T-y_f)=v$; finally, as $v\in [D(A),H]_{1/2}$, 
the vector $u_T-y_f$ fulfils the condition in  \eqref{Y1-id}. Hence $(f_n,u_{T,n})$ converges in the norm of $Y_1$ to the element $(f,u_T)$ in $Y_1$, as desired.

To compare with Theorem~\ref{fvp-thm}, note that there clearly are continuous embeddings 
\begin{equation}
  \label{XY-imb}
  X_1\hookrightarrow X, \qquad Y_1\hookrightarrow Y.
\end{equation}
Thus prepared, it is now possible to give a concise proof of the following novel result, which is
a companion to Theorem~\ref{fvp-thm} in which the solutions have regularity properties that
(instead of relating to the extension $A\in\B(V,V^*)$ as in Theorem~\ref{fvp-thm}) are more closely
connected to the unbounded operator $A$ in $H$: 

\begin{theorem} \label{fvp1-thm}
  Let $A$ be a $V$-\emph{coercive} Lax--Milgram operator defined from a triple $(H,V,a)$ as above,
  and let $(f,u_T)\in L_2(0,T;H)\oplus [D(A),H]_{1/2}$ be given.
  Then the abstract final value problem \eqref{fvA-intro} has a solution $u(t)$ belonging the space
  $X_1$ in \eqref{X1-id}, if and only if the data $(f,u_T)$ belong to the subspace $Y_1$, that is, 
  \begin{equation} \label{cc1-id}
     u_T - y_f \in D(e^{TA}; [D(A),H]_{1/2}).
  \end{equation}
In the affirmative case, the solution $u$ is uniquely determined in $X_1$ and it fulfils 
$\|u\|_{X_1} \le   c \|(f,u_T)\|_{Y_1}$,
whence the solution operator $\cal{R}_1\colon (f,u_T)\mapsto u$ is continuous $Y_1\to X_1$. Moreover,
\begin{equation} \label{ut-id}
  u(t) = e^{-tA}e^{TA}\Big(u_T-\int_0^T e^{-(T-t)A}f(t)\,dt\Big) + \int_0^t e^{-(t-s)A}f(s) \,ds,
\end{equation}
where all terms belong to $X_1$ as functions of $t\in[0,T]$, and the difference in
\eqref{cc1-id} equals $e^{-TA}u(0)$, which belongs to $D(e^{TA}; [D(A),H]_{1/2})=e^{-TA}([D(A),H]_{1/2})$. 
\end{theorem}

\begin{proof}
If \eqref{fvA-intro} is solved by $u \in X_1$, then $u(T)=u_T$ is by \eqref{u0uT-id} reached from
the unique initial state $u(0)$, which is in $[D(A),H]_{1/2}$ since $u$ as a member of $X_1$ is continuous 
$[0,T]\mapsto [D(A),H]_{1/2}$. But then we have
$u_T-y_f = e^{-TA} u(0)\in e^{-TA}([D(A),H]_{1/2})$, whence necessity of \eqref{cc1-id} and the
last claim is covered.

Given $(f,u_T)\in Y_1$, there is first of all by Theorem~\ref{fvp-thm} and \eqref{XY-imb} a unique
solution $u\in X$ satisfying \eqref{ut-id}.
Secondly, \eqref{cc1-id} entails that the implicit initial state $u(0)$ belongs to $[D(A),H]_{1/2}$, and since $f$ is given in $L^2(0,T;H)$, Theorem~\ref{HDA-thm} then yields the stronger conclusion that $u$ as
well as all terms in \eqref{ut-id} belong to the subspace $X_1\subset X$.
The stated estimate $\|u\|_{X_1} \le   c \|(f,u_T)\|_{Y_1}$ may in view of \eqref{CDAH-ineq} be deduced  from the one in
Theorem~\ref{HDA-thm} by replacing $u_0$ by the expression for $u(0)$, and adding $\|u_T\|^2_{[D(A),H]_{1/2}}$.
\end{proof}

It is noted that the above proof of existence of a solution in $X_1$ was conducted via an a posteriori estimate of the solution $u$ belonging to  the larger space $X$. Therefore Theorem~\ref{fvp1-thm} is basically a regularity result.

To extend the control and operator theoretic remarks from $Y$ to $Y_1$ in \eqref{Y1-id}, one may as
a variant of \eqref{EFi-id} consider the unbounded composite operator  
\begin{equation} \label{EFi'-id}
  L_2(0,T; H)\oplus [D(A),H]_{1/2} \xrightarrow[\;]{\quad \Phi\quad} [D(A),H]_{1/2}
  \xrightarrow[\;]{\quad e^{TA}\quad} [D(A),H]_{1/2}.
\end{equation}
This map $e^{TA}\Phi$ is defined at $(f,u_T)$ if and only if \eqref{cc1-id} is met, so $Y_1$ is its domain; and clearly $\|\cdot\|_{Y_1}$ is its graph norm.  
To obtain completeness of $Y_1$ in this set-up, it therefore suffices to show that the above 
$e^{TA}\Phi$ is closed; which by the continuity of $\Phi(f,u_T)=u_T-y_f$ results from closedness of the restriction in \eqref{EFi'-id} of $e^{TA}$ to an operator in $[D(A),H]_{1/2}$, but that follows at once from its closedness in $H$. Since $[D(A),H]_{1/2}$ has a Hilbert space structure (being the graph normed domain $D(\Lambda^{1/2})$ for a (non-unique) positive selfadjoint operator $\Lambda$ in $H$, cf.\ \cite[Sect.\ I.2]{LiMa72}), the data space $Y_1$ is also Hilbertable.

In analogy with Remark~\ref{P-rem}, it is seen that $u\mapsto (u'+Au,u(T))$ also gives a bounded operator
$\cal{P}_1\colon X_1\to Y_1$ and that the solution operator $\cal{R}_1\colon Y_1\to X_1$ is everywhere defined and bounded according to Theorem~\ref{fvp1-thm}; and moreover that $\cal{P}_1\cal{R}_1=I$ and $\cal{R}_1\cal{P}_1=I$ hold on $Y_1$ and $X_1$, respectively. This can be summed up thus:

\begin{corollary} \label{fvA-cor}
  The final value problem \eqref{fvA-intro} generated by the Lax--Milgram operator $A$, 
  defined from a $V$-coercive triple $(H,V,a)$, is 
  \emph{isomorphically} well posed in the pair of spaces $(X,Y)$ as well as in $(X_1,Y_1)$.
\end{corollary}

In addition to the above isomorphic well-posedness, it is remarked that the Duhamel formula \eqref{u-id} also shows that $u(T)$ has two 
radically different contributions, even if $A$ has nice properties.

First, for $t=T$ the integral in \eqref{u-id} amounts to $y_f$, which can be
\emph{anywhere} in $H$, as
$f\mapsto y_f$ is a \emph{surjection} $y_f\colon L_2(0,T;V^*)\to H$. 
This was shown for $k=0$ via the Closed Range Theorem in \cite[Prop.\ 5]{ChJo18ax}, 
and more generally the surjectivity follows from this case since 
$e^{-(T-s)A}f(s)=e^{-(T-s)(A+kI)}e^{(T-s)k}f(s)$ in the integrand, whereby $A+kI$ is $V$-elliptic and $f\mapsto
e^{(T-\,\cdot)k}f$ is a bijection on $L_2(0,T;V^*)$.

Secondly, in \eqref{u-id} the first term $e^{-tA}u(0)$ solves $u'+Au=0$, and for $u(0)\ne0$ there is for  $V$-\emph{elliptic}  $A$
the precise property in non-selfadjoint dynamics that the ``height'' function $h(t)= |e^{-tA}u(0)|$ is
\begin{equation}  \label{strictly-eq}
  \text{strictly positive ($h>0$)},\quad \text{strictly decreasing ($h'<0$)},\quad \text{and \emph{strictly convex}
    ($\Leftarrow h''>0$)}  .
\end{equation}
Whilst this holds if $A$ is self-adjoint or normal, it was emphasized in \cite{ChJo18ax}
that it suffices that $A$ is just hyponormal (i.e., $D(A)\subset D(A^*)$ and $|Ax|\ge|A^*x|$ for
$x\in D(A)$, following Janas \cite{Jan94}). Recently this was followed up in \cite{JJ18logconv, JJ20}, 
where the (in the context) stronger 
\emph{logarithmic} convexity of $h(t)$ was proved \emph{equivalent} to the property weaker than hyponormality of $A$
that for $x\in D(A^2)$,
\begin{equation} \label{Alogconv-id}
  2(\Re\scal{A x}{x})^2\le \Re\scal{A^2x}{x}|x|^2+|A x|^2|x|^2 .
\end{equation}

For $V$-\emph{coercive} $A$ the strict positivity
$h>0$ also holds, by injectivity of $e^{-tA}$. But the strict decay need not extend to such $A$ (e.g.\ $h$ is constant if $u_0$ is a constant function in $\Omega$ for $A=-\lap_N$; cf.\  Section~\ref{Neu-sect}).
However, most conveniently, the strict convexity in \eqref{strictly-eq} can simply be replaced by log-convexity for coercive $A$.

Indeed, the characterisation in \cite[Lem.\ 2.2]{JJ18logconv} or \cite{JJ20}
of the log-convex $C^2$-functions $f(t)$ on $[0,\infty[\,$ as the solutions of
the differential inequality $f''\cdot f\ge (f')^2$  and the resulting criterion for $A$ in
\eqref{Alogconv-id} apply \emph{verbatim} to the coercive case. Hereby the differential calculus in
Banach spaces is exploited in a classical derivation of  the formulae for $u(t)=e^{-tA}u(0)$,
\begin{equation}
  h'(t)=-\frac{\Re\scal{Au(t)}{u(t)}}{|u(t)|},
\qquad
  h''(t)=\frac{\Re\scal{A^2u(t)}{u(t)}+|Au(t)|^2}{|u(t)|} 
            -\frac{(\Re\scal{Au(t)}{u(t)})^2}{|u(t)|^3}.
\end{equation}
But it is due to the strict positivity $|u(t)|>0$ for $t\ge0$ in the denominators that
the expressions make sense, so Proposition~\ref{inj-prop} is also crucial here.
The singularity of $|\cdot|$ at the origin likewise poses no problems for 
differentiation of $h(t)$. So perhaps the natural formulae for $h'$, $h''$ were first made rigorous
in \cite{JJ18logconv}. 

However, the stiffness intrinsic to \emph{strict} convexity corresponds well (when applicable) with the fact that
$u(T)=e^{-TA}u(0)$ also for coercive $A$  is confined to a very small dense space, as by the analyticity
\begin{equation}
  \label{DAn-cnd}
  u(T)\in \textstyle{\bigcap_{n\in\N}} D(A^n).
\end{equation}
For $u'+Au=f\ne0$, the possible 
$u_T$ will hence be a sum of an arbitrary $y_f\in H$ and 
the stiff term $e^{-TA}u(0)$. Thus $u_T$ can be prescribed in the affine space
$y_f+D(e^{TA})$. As the vector  $y_f\ne0$ will shift $D(e^{TA})\subset H$
in an arbitrary direction, $u(T)$ can be expected \emph{anywhere} in $H$ (unless $y_f\in D(e^{TA})$ is known).
So neither \eqref{DAn-cnd} nor $u(T)\in D(e^{TA})$ can be expected to hold if $y_f\ne0$---not even
if $|y_f|$ is much smaller than $|e^{-TA}u(0)|$. Thus it seems fruitful for final value problems to
consider inhomogeneous equations from the outset.

\begin{remark} \label{GS-rem}
Grubb and Solonnikov~\cite{GrSo90} treated
a large class of \emph{initial}-boundary problems of parabolic pseudo-differential equations
and worked out compatibility conditions characterising the
well-posedness in full scales of anisotropic $L_2$-Sobolev spaces
(such conditions have a long history in the differential operator case, cf.\ work
of Lions and Magenes \cite{LiMa72} and Ladyzhenskaya, Solonnikov and Ural'ceva \cite{LaSoUr68}).
Their conditions are local at the curved corner $\{0\}\times\Gamma$, except for 
half-integer values of the smoothness $s$ that were shown to require coincidence, which 
is expressed via integrals over the Cartesian product of the two boundaries $\{0\}\times\Omega$ and
$\,]0,T[\,\times\,\Gamma$.
While the conditions in \cite{GrSo90} address the regularity, 
condition \eqref{eq:cc-intro} in Theorem~\ref{fvp-thm} and \eqref{cc1-id} in Theorem~\ref{fvp1-thm}
pertain to the existence of solutions in the respective spaces.
\end{remark}

\begin{remark} Recently
Almog, Grebenkov, Helffer, Henry \cite{AlHe15,GrHelHen17,GrHe17} studied
the complex Airy operator $\mlap+\im x_1$  via triples $(H,V,a)$, leading to Dirichlet, Neumann, Robin and
transmission boundary conditions, in bounded and unbounded regions. To improve earlier remarks,
Theorem~\ref{fvp-thm} is expected to apply to their Dirichlet realisations while Theorem~\ref{fvp1-thm} would pertain to the Neumann and Robin realisations, leading to final value problems for those of their
realisations that satisfy the (strong) coercivity in \eqref{coerciv-id}. As
$\mlap+\im x_1$ has empty spectrum on $\Rn$, as shown by Herbst \cite{Her79}, it
remains to be clarified for which of the regions in \cite{AlHe15,GrHelHen17,GrHe17} there is a strictly descending chain of domains as in \eqref{dom-intro}.
\end{remark}

\section{The heat problem with the Neumann condition}
  \label{Neu-sect}

In the sequel $\Omega$ stands for a $C^\infty$ smooth, open bounded set in $\Rn$,
$n\ge1$ as described in \cite[App.~C]{G09}. In particular $\Omega$  is
locally on one side of its boundary $\Gamma=\partial \Omega$.
The problem is  then to characterise the $u(t,x)$ such that 
\begin{equation}  \label{heatN_fvp}
\left.
\begin{aligned}
  \partial_tu(t,x) -\Delta u(t,x) &= f(t,x) &&\text{ in } \, ]0,T[ \times \Omega ,
\\
  \gamma_1 u(t,x) &= 0 && \text{ on } \, ]0,T[\, \times \Gamma,
\\
   r_T u(x) &= u_T(x) && \text{ at } \left\{ T \right\} \times \Omega.
\end{aligned}
\right\}
\end{equation}
While $r_Tu(x)=u(T,x)$,
the Neumann trace on $\Gamma$ is written in the operator notation $\gamma_1u=
(\nu\cdot\nabla u)|_{\Gamma}$, whereby $\nu$ is the unit outward pointing normal vector field.
Similarly $\gamma_1$ is used for traces on $\, ]0,T[\, \times \Gamma$.

Moreover, $H^m(\Omega)$ denotes the Sobolev space normed by $\|u\|_m =
\big(\sum_{|\alpha|\le m}\int_\Omega |\partial^\alpha u|^2\,dx\big)^{1/2}$, for $m\in\N_0$, which up to equivalent 
norms equals the space $H^{m}(\overline{\Omega})$ of 
restrictions to $\Omega$ of $H^{m}(\Rn)$ endowed with the infimum norm, which also is denoted by $\|\cdot\|_m$. 
This is useful since the dual space of $H^m(\overline\Omega)$ has an identification with the closed
subspace of $H^{-m}(\Rn)$ that is given by the support condition in   
\begin{equation}
  H^{-m}_0(\overline{\Omega})=\Set{ u\in H^{-m}(\Rn)}{\operatorname{supp} u\subset \overline{\Omega}}.
\end{equation} 
For these standard facts in functional analysis the reader may consult \cite[App.\ B.2]{H}.
Chapter 6 and (9.25) in \cite{G09} could also be references for this and basics
on boundary value problems; cf.\ also \cite{Eva10, Rau91}.

The main result in Theorem~\ref{fvp-thm} applies to \eqref{heatN_fvp} for 
$V= H^1(\overline\Omega)$,  $H = L_2(\Omega)$ and $V^* \simeq H^{-1}_0(\overline{\Omega})$, for
which there are inclusions 
\begin{equation}
  H^1(\overline\Omega)\subset L_2(\Omega)\subset H^{-1}_0(\overline{\Omega}),   
\end{equation}
when $g\in L_2(\Omega)$ via extension by zero outside of $\Omega$, denoted by $e_\Omega$,  is identified with 
$e_\Omega g$ belonging to $H^{-1}_0(\overline{\Omega})$. (This of course modifies the usual identification 
$L_2(\Omega)^*\simeq L_2(\Omega)$ slightly, but $e_\Omega g$ is the function on $\Rn$ at which the infimum in the 
norm $\|g\|_0$ is attained.)
The Dirichlet form 
\begin{align}  \label{sform-id}
  s(u,v) = \sum_{j=1}^n \scal{\partial_j u}{\partial_j v}_{L_2(\Omega)}
         = \sum_{j=1}^n \int_\Omega {\partial_j u}\overline{\partial_j v}\, dx 
\end{align}
fulfils $|s(v,w)|\le \|v\|_1\|w\|_1$, and the coercivity
 \eqref{coerciv-id} holds for $C_4=1$, $k=1$ since $s(v,v)=\|v\|_1^2-\|v\|_0^2$.

The induced Lax--Milgram operator is the Neumann realisation $\mlap_N$, which is selfadjoint due
to the symmetry of $s$ and has
its domain given by $D(\lap_N)=\Set{u\in H^2(\overline{\Omega})}{\gamma_1 u=0}$. 
This is a non-trivial classical result (cf.\ the remarks prior to Theorem 4.28 in \cite{G09}, or Section
11.3 ff.\ there; or \cite{Rau91}).
Thus the homogeneous boundary condition is imposed via the condition $u(t)\in D(\lap_N)$ for $t$ in
$\,]0,T[\,$ a.e.

By the $H^1$-coercivity, $-A = \lap_N$ generates an analytic semigroup of injections
$e^{z\lap_N}$ in $\B(L_2(\Omega))$, cf.\ Proposition~\ref{inj-prop} and Remark~\ref{generator-rem}, and like before $e^{-t\lap_N}:=(e^{t\lap_N})^{-1}$. 
The extension 
$\tilde\lap\in \B( H^{1}(\overline\Omega), H^{-1}_0(\overline{\Omega}))
$ 
induces the analytic semigroup $e^{z\tilde\lap}$ defined for $z\in S_{\pi/4}$ on
$H^{-1}_0(\overline{\Omega})$, and as observed in \cite{JJ19cor}, if not before, it can be
explicitly described:

\begin{lemma}
The action of the bounded extension $\tilde \lap\colon H^{1}(\overline\Omega) \rightarrow H^{-1}_0(\overline{\Omega})$ of $\lap_N$ is given by 
\begin{align} \label{tildeLap-id}
  \tilde\lap u&=\dv(e_\Omega\grad u) \qquad \text{for $u\in H^{1}(\overline\Omega)$},  
\\  \label{tildeLap'-id}
    \tilde\lap u&=e_\Omega(\lap u)-(\gamma_1 u)dS \qquad \text{for $u\in H^{2}(\overline\Omega)$},
\end{align}
whereby $dS\in\cal{D}'(\Rn)$ denotes the surface measure at $\Gamma$.
\end{lemma}
\begin{proof}
 When $w\in H^{1}(\Rn)$ coincides with $v$ in $\Omega$, for given $u$, $v\in
 H^{1}(\overline\Omega)$, then \eqref{sform-id} gives \eqref{tildeLap-id} as follows,
\begin{equation}
  \begin{split}
    \dual{-\tilde\lap u}{v}=s(u,v)
   & = \sum_{j=1}^n \int_{\Rn} e_\Omega(\partial_j u)\cdot\overline{\partial_j w}\, dx
\\
   & =\sum_{j=1}^n\dual{-\partial_j(e_\Omega\partial_j u)}{w}_{H^{-1}(\Rn)\times
    H^{1}(\Rn)}
  =\dual{-\dv(e_\Omega\grad u)}{v}_{H^{-1}_0(\overline{\Omega})\times H^{1}(\overline\Omega)}.
  \end{split}
\end{equation}
To show \eqref{tildeLap'-id}, one may recall that $\partial_j(u\chi_\Omega)=(\partial_j
u)\chi_\Omega-\nu_j(\gamma_0u)dS$ holds for $u\in C^{1}(\Rn)$ when $\chi_\Omega$
denotes the characteristic function of $\Omega$, and $\gamma_0$ stands for the restriction to
$\Gamma$; cf.\ the  proof in \cite[Thm.\ 3.1.9]{H}.  
Replacing $u$ by $\partial_j u$ for some $u\in C^{2}(\overline\Omega)$, and using that $\nu(x)$ is
a smooth vector field around $\Gamma$, we get 
$\partial_j(e_\Omega\partial_ju)=e_\Omega(\partial_j^{2} u)-(\gamma_0\nu_j\partial_j u)dS$, which
after summation over $j$ yields \eqref{tildeLap'-id} for such $u$ in view of \eqref{tildeLap-id}.
The formula then extends to all $u\in H^{2}(\overline\Omega)$ by continuity and density of $C^2$.
\end{proof}

In \eqref{tildeLap'-id} the last term vanishes  
for $u\in D(\lap_N)$ as $\gamma_1u=0$; whence for such $u$, clearly $\tilde\lap u
=\dv (e_\Omega\grad u)$ identifies  in $\Omega$  with the $L_2$-function $\lap u$.
However, for general $u$ in the form domain $H^{1}(\overline\Omega)$,
the terms on the right hand side of \eqref{tildeLap'-id} do not make sense.

To account for the consequences of Theorem~\ref{fvp-thm} for \eqref{heatN_fvp}, note that \eqref{eq:X}
gives rise to the solution space 
\begin{align}
  X_0 &= L_2(0,T;H^1(\overline{\Omega})) \bigcap C([0,T]; L_2(\Omega)) \bigcap H^1(0,T; H^{-1}_0(\overline{\Omega})), 
   \label{X0-id}
\\
  \|u\|_{X_0}&= \big(\int_0^T\|u(t)\|^2_{H^{1}(\overline{\Omega})}\,dt 
                 +\sup_{t\in[0,T]}\int_\Omega |u(x,t)|^2\,dx+
                 \int_0^T\|\partial_t u(t)\|^2_{H^{-1}_0(\overline{\Omega})}\,dt  \Big)^{1/2}.
\end{align}
The corresponding data space is here given in terms of the vector $y_f=\int_0^T e^{(T-t)\lap}f(t)\,dt$
from \eqref{yf-eq} as 
\begin{align}
  Y_0&= \left\{ (f,u_T) \in L_2(0,T;H^{-1}_0(\overline{\Omega})) \oplus L_2(\Omega) \Bigm|  
                  u_T - y_f \in D(e^{-T\lap_N}) \right\},
\label{Y0-id}
\\
 \| (f,u_T) \|_{Y_0}  
  &= \Big(\int_0^T\|f(t)\|^2_{H^{-1}_0(\overline{\Omega})}\,dt 
  + \int_\Omega\big(|u_T(x)|^2+|e^{-T\lap_N}(u_T - y_f )(x)|^2\big)\,dx\Big)^{1/2}.
\end{align} 
Using this framework, as in \cite[Thm.\ 4.1]{JJ19cor}, the above Theorem~\ref{fvp-thm} at once
gives the following (partial) result for \eqref{heatN_fvp}, which  
further below may serve as a reference point for the reader:

\begin{theorem}  \label{heatN-thm}
Let $A=\mlap_N$ be the Neumann realization of the Laplacian in $L_2(\Omega)$ and having its bounded extension
$H^1(\overline{\Omega})\to H^{-1}_0(\overline{\Omega})$ given by $-\tilde\lap=-\dv(e_\Omega\grad\cdot)$.
Whenever $f \in L_2(0,T;H^{-1}_0(\overline{\Omega}))$, $u_T \in L_2(\Omega)$, 
there exists a solution $u$ in $X_0$ of 
\begin{equation}
  \partial_t u-\dv(e_\Omega\grad u)=f,\qquad  r_Tu=u_T
\end{equation}
if and only if the data $(f,u_T)$ belong to $Y_0$, i.e.\ if and only if
\begin{equation}  \label{heat-ccc}
  u_T - \int_0^T e^{(T-s)\tilde\lap}f(s) \,ds\quad \text{ belongs to }\quad D(e^{-T \lap_N})=R(e^{T\lap_N}). 
\end{equation}
In the affirmative case, $u$ is uniquely determined in $X_0$ and satisfies the estimate
$\|u\|_{X_0} \leq c \| (f,u_T) \|_{Y_0}$.
It is given by the formula, in which all terms belong to $X_0$,
\begin{equation} 
  u(t) = e^{t\lap_N}e^{-T\lap_N}\Big(u_T-\int_0^T e^{(T-t)\tilde\lap}f(t)\,dt\Big) + \int_0^t e^{(t-s)\tilde\lap}f(s) \,ds.
\end{equation}
Moreover, in \eqref{heat-ccc} the difference   equals 
$e^{T\lap_N}u(0)$ in $L_2(\Omega)$. 
\end{theorem}

Besides the deplorable fact that $\tilde\lap=\dv(e_\Omega\grad\cdot)$ appears in the differential equation,
instead of $\lap$, there is also no information on the boundary condition.
However, if  in addition \eqref{Holder-id} is fulfilled, the H{\"o}lder continuity yields $u(t)\in
D(\lap_N)$ for $t>0$, so $\gamma_1u=0$ is
fulfilled and $\tilde\lap u$ identifies with $\lap u$; whence one has

\begin{corollary}[\cite{JJ19cor}]
  If $u_T\in L_2(\Omega)$ and $f\colon\,[0,T]\to L_2(\Omega)$ is H{\"o}lder continuous of 
  order $\sigma\in\,]0,1[\,$, and if $u_T-y_f$ fulfils  \eqref{heat-ccc}, then the
  homogeneous Neumann heat conduction final value problem \eqref{heatN_fvp} has a uniquely
  determined solution $u$ in $X_0$, satisfying $u(t)\in \Set{u\in H^2(\overline{\Omega})}{\gamma_1 u=0}$ for
  $t>0$, and depending continuously on $(f,u_T)$ in $Y_0$. Hence problem \eqref{heatN_fvp} is well posed in the
  sense of Hadamard. 
\end{corollary}

This result is less than ideal, of course, since H{\"o}lder continuity is not available for the
general source terms $f$ in $Y_0$, and consequently the corollary pertains only to some dense, but non-closed 
subspace of $Y_0$, which is unsatisfying when, as stated, the stability only refers to the norm on the full data space $Y_0$.

It is therefore the purpose of this paper to obtain \emph{isomorphic} well-posedness of \eqref{heatN_fvp} in
other, more suitable spaces $X_1$, $Y_1$.
The point of departure is the general well-posedness 
result in Theorem~\ref{fvp1-thm}, whereby the interpolation space satisfies 
$[D(\lap_N),L_2(\Omega)]_{1/2}=V= H^1(\overline{\Omega})$ here, as $\lap_N$ is self-adjoint in
$L_2(\Omega)$.

In view of this, \eqref{DSU-id} yields for the inverse $e^{-T\lap_N}$ that $D(e^{-T\lap_N};H^1(\overline{\Omega}))=
e^{T\lap_N}(H^1(\overline{\Omega}))$.
The data space $Y_1$ in \eqref{Y1-id} is therefore taken, in terms of $y_f=\int_0^T e^{(T-t)\lap}f(t)\,dt$ belonging to
$H^1(\overline{\Omega})$, as
\begin{equation}  \label{Y1'-id}
\begin{split}
  Y_1&= \left\{ (f,u_T) \in L_2(0,T;L_2({\Omega})) \oplus H^1(\overline{\Omega}) \Bigm|  
                  u_T - y_f \in D(e^{-T\lap_N}; H^1(\overline{\Omega})) \right\},
\\
 \| (f,u_T) \|_{Y_1}  
  &= \Big(\int_0^T\|f(t)\|^2_{L_2({\Omega})}\,dt 
  + \int_\Omega \sum_{|\alpha|\le1}\big(|\partial^\alpha_x u_T(x)|^2
  +|\partial^\alpha_x e^{-T\lap_N}(u_T - y_f )(x)|^2\big)\,dx\Big)^{1/2}.
\end{split}
\end{equation}
Correspondingly the solution space in \eqref{X1-id} amounts to
\begin{equation}     \label{X1'-id}
\begin{split}
  X_1 &= L_2(0,T;H^2(\overline{\Omega})) \bigcap C([0,T]; H^1(\overline{\Omega})) \bigcap H^1(0,T; L_2(\Omega)),
\\
  \|u\|_{X_1}&= \big(\int_0^T\|u(t)\|^2_{H^{2}(\overline{\Omega})}\,dt 
                 +\sup_{t\in[0,T]}\sum_{|\alpha|\le 1} \int_\Omega |\partial^\alpha_x u(x,t)|^2\,dx+
                 \int_0^T\|\partial_t u(t)\|^2_{L_2({\Omega})}\,dt  \Big)^{1/2}.
\end{split}
\end{equation}
There are, of course, also continuous embeddings $X_1\hookrightarrow X_0$ and  $Y_1\hookrightarrow Y_0$ among these spaces.

Within this framework, the stronger Theorem~\ref{fvp1-thm} at once gives the following novelty 
for the classical inverse heat conduction problem with the homogeneous Neumann condition in \eqref{heatN_fvp}:

\begin{theorem}  \label{heatN'-thm}
Let $A=\mlap_N$ be the Neumann realization of the Laplacian in $L_2(\Omega)$.
If $f \in L_2(0,T;L_2(\Omega))$ and $u_T \in H^1(\overline{\Omega})$, 
there exists in the Banach space $X_1$ in \eqref{X1'-id} a solution $u$  of the final value problem \eqref{heatN_fvp}, namely
\begin{equation}
  \partial_t u-\lap_N u=f,\qquad \gamma_1u=0,\qquad  r_Tu=u_T,
\end{equation}
if and only if the data $(f,u_T)$ are given in the Banach space $Y_1$ in \eqref{Y1'-id}, i.e.\ if and only if
$(f,u_T)$ satisfy the compatibility condition:
\begin{equation}  \label{heat'-ccc}
  u_T - \int_0^T e^{(T-s)\lap_N}f(s) \,ds\quad \text{ belongs to }\quad D(e^{-T \lap_N},H^1(\overline{\Omega})). 
\end{equation}
In the affirmative case, the solution $u$ is uniquely determined in $X_1$ and for some constant $c>0$ independent of
$(f,u_T)$ it satisfies $\|u\|_{X_1} \leq c \| (f,u_T) \|_{Y_1}$.
It is given by the formula, in which all terms belong to $X_1$,
\begin{equation} \label{heatDuhamel-id}
  u(t) = e^{t\lap_N}e^{-T\lap_N}\Big(u_T-\int_0^T e^{(T-t)\lap_N}f(t)\,dt\Big) + \int_0^t e^{(t-s)\lap_N}f(s) \,ds.
\end{equation}
Furthermore the difference  in \eqref{heat'-ccc} equals 
$e^{T\lap_N}u(0)$, which belongs to $D(e^{-T \lap_N},H^1(\overline{\Omega}))=e^{T\lap_N}(H^1(\overline{\Omega}))$. 
\end{theorem}

To emphasize the complex nature  of the inverse Neumann heat equation, it might serve a purpose to write out the
inequality that  according to Theorem~\ref{heatN'-thm} is satisfied by the solution $u$ in $X_1$:
\begin{equation}
\begin{split}
  &\qquad\int_0^T\int_\Omega\big(|\partial_tu(t,x)|^2 +\sum_{|\alpha|\le2}|\partial^\alpha_x u(t,x)|^2\big)\,dx\,dt 
          +\sup_{t\in[0,T]}\sum_{|\alpha|\le 1} \int_\Omega |\partial^\alpha_x u(x,t)|^2\,dx 
\\
&\le
  c\int_0^T\int_\Omega |f(t,x)|^2\,dx\,dt + c\sum_{|\alpha|\le1}\int_\Omega\Big(|\partial^\alpha_x u_T(x)|^2
  +\big|\partial^\alpha_x 
                       e^{-T\lap_N}(u_T - \int_0^T e^{(T-t)\lap_N}f(t)\,dt )(x)\big|^2\Big)\,dx. 
\end{split}  
\end{equation}

As a particular case of the comments after Theorem~\ref{fvp1-thm}, $\cal{P}=(\partial_t-\lap_N,r_T)$ is a linear homeomorphism $X_1\to Y_1$ between Hilbertable spaces $X_1$, $Y_1$. 
Hence there is the following new result on a classical problem:

\begin{corollary} \label{heatN-cor}
  The final value problem \eqref{heatN_fvp} for the homogeneous Neumann heat equation in the smooth open bounded set $\Omega$ is isomorphically well posed in the spaces $X_1$ and $Y_1$ in \eqref{X1'-id} and \eqref{Y1'-id}.
\end{corollary}
It is left for the future to develop a theory for the final value heat conduction problem subjected
to the inhomogeneous Neumann condition $\gamma_1u=(\nu\cdot\nabla u)|_\Gamma =g$ at the curved
boundary. It is envisaged that the techniques used in \cite{ChJo18ax} for the
inhomogeneous Dirichlet condition can be adapted to the set-up above.

\begin{remark}
To give some background, it is recalled that there is a phenomenon of $L_2$-instability in case
$f=0$ in \eqref{heat-intro}. This was perhaps first described by
Miranker  \cite{Mir61}, who addressed the homogeneous Dirichlet condition at the boundary. The
instability is found via the Dirichlet realization 
of the Laplacian, $\mlap_D$, and its $L_2(\Omega)$-orthonormal
basis $e_1(x), e_2(x), \dots$ of eigenfunctions associated to the
usual ordering of its eigenvalues $0<\lambda_1\le\lambda_2\le\dots$ counted with multiplicities. A similar notation applies to
the Neumann realisation $\mlap_N$ studied above, although $\lambda_1=0$ in this case.
It has been a major classical theme (with a too rich history to recall here) that 
Weyl's law for the counting function $N(\lambda)=\#\set{j}{0\le\lambda_j\le\lambda}$ in terms of the
measures $|\Omega|$, $|\partial\Omega|$ and the volume $\omega_n$ of the unit ball in $\Rn$
fulfils, for $\lambda\to\infty$,
\begin{equation} 
 N(\lambda)= (\frac{\sqrt{\lambda}}{2\pi})^n\omega_n|\Omega|\mp \frac14
 (\frac{\sqrt{\lambda}}{2\pi})^{n-1}\omega_{n-1}|\partial\Omega| + o(\lambda^{(n-1)/2}).
\end{equation}
Hereby $-$ and $+$ refers to the Dirichlet and Neumann boundary conditions, respectively, but as the
leading term is the same, a classical  inversion gives the same crude eigenvalue
asymptotics for both conditions,
\begin{equation} \label{Weyl-id}
  \lambda_j={\cal O}(j^{2/n})\quad\text{ for $j\to\infty$}.
\end{equation}
Hence there is also $L_2$-instability for the homogeneous Neumann problem: 
The eigenfunction basis $e_1(x)$, $e_2(x), \dots$ gives rise to a sequence of final value data
$u_{T,j}(x)=e_j(x)$ lying on the unit sphere in $L_2(\Omega)$ as
$\|u_{T,j}\|= \|e_j\|=1$ for $j\in\N$.
But the corresponding solutions to $u'\mlap u=0$,
i.e.\  $u_j(t,x)=e^{(T-t)\lambda_j}e_j(x)$,
have initial states $u(0,x)$ with $L_2$-norms that because of \eqref{Weyl-id} 
grow \emph{rapidly} for $j\to\infty$,
\begin{equation}
  \|u_j(0,\cdot)\| = e^{T\lambda_j}\|e_j\| = e^{T\lambda_j}\nearrow\infty.
\end{equation}
However, this $L_2$-instability
only indicates that the $L_2(\Omega)$-norm is an insensitive choice for problem \eqref{heat-intro}. The
task is hence to obtain a norm on $u_T$ giving better control over the backward
calculations of $u(t,x)$---for the homogeneous Neumann heat problem \eqref{heat-intro}, 
an account of this was given in Theorem~\ref{heatN'-thm} ff.
\end{remark}

\section{Final remarks}  \label{final-sect}
\begin{remark}
  Since the Neumann condition $\gamma_1u=0$ is given in terms of a trace operator effectively of class~2
  (as $\gamma_1$ is defined on $H^2$ but not on $H^1$), it is not surprising that the well-posedness for \eqref{heatN_fvp} 
is obtained in the more regular spaces $X_1$,  $Y_1$ in Theorem~\ref{heatN'-thm}, whereas Theorem~\ref{heatN-thm} is 
somewhat inconclusive. However, since Theorem~\ref{heatN'-thm} can be seen as a regularity result adjoined to 
Theorem~\ref{heatN-thm}, it is an important clarification that the additional assumption $f\in
L_2(0,T;L_2(\Omega))$ does not alone suffice for the regularity needed to ascertain that $u$
belongs to $X_1$: to avoid a singularity at $t=0$ in the $L_2$-norm of $\lap_Ne^{t\lap_N}u(0)$, the
implicit initial state $u(0)$ must be stipulated to belong to the interpolation space
$[D(\lap_N),L_2]_{1/2}=H^1(\Omega)$; cf.\ the equivalent conditions in Theorem~\ref{LM101-thm}. In
particular this gave rise to the  compatibility condition \eqref{heat'-ccc} using the modified
domain $D(e^{-T\lap_N},H^1(\Omega))$, which is a new and  non-trivial element of the theory. 
\end{remark}

\begin{remark}
It is envisaged that the isomorphic well-posedness in Theorem~\ref{heatN'-thm} and Corollary~\ref{heatN-cor} can be carried over to the inhomogeneous Neumann condition and to the Robin condition and other class 2 problems. More generally the present results should extend to final value problems for  differential equations with boundary conditions that define parabolic Cauchy problems belonging to the pseudo-differential boundary operator calculus, cf.\  \cite{GrSo90, G96}.
This is left for the future---the main purpose of the present paper is to show how the compatibility conditions should be modified in order to cover a problem of a high class.
\end{remark}

\begin{remark} \label{LM-rem}
Injectivity of the linear map $u(0)\mapsto u(T)$ for the homogeneneous equation $u'+Au=0$,
i.e.\ its backwards uniqueness, was proved 60 years ago by Lions and
Malgrange~\cite{LiMl60} for problems with $t$-dependent
sesquilinear forms $a(t;u,v)$. Besides some
$C^1$-properties in $t$, they assumed that (the principal part of) $a(t;u,v)$ is
symmetric and uniformly $V$-coercive in the sense that $a(t;v,v)+\lambda\|v\|^2\ge \alpha\|v\|^2$
for fixed  $\lambda\in\R$, $\alpha>0$ and all $v\in V$. 
(Bardos and Tartar \cite{BaTa73} relaxed these $C^1$-assumptions and made some non-linear extensions.)
In Problem~3.4 of \cite{LiMl60}, the authors
asked if backward uniqueness can be shown under the general \emph{non-symmetric} hypothesis of strong $V$-coercivity
$\Re a(t;v,v)+\lambda\|v\|^2\ge \alpha\|v\|^2$. 
The above Proposition~\ref{inj-prop} gives an affirmative
answer for the $t$-independent case of their problem.
\end{remark}

\section*{Acknowledgement}
The author is grateful for some references that were provided by the anonymous referee.

\providecommand{\bysame}{\leavevmode\hbox to3em{\hrulefill}\thinspace}
\providecommand{\MR}{\relax\ifhmode\unskip\space\fi MR }
\providecommand{\MRhref}[2]{%
  \href{http://www.ams.org/mathscinet-getitem?mr=#1}{#2}
}
\providecommand{\href}[2]{#2}

\end{document}